\documentclass[11pt,a4paper]{article}
\usepackage{mathrsfs}
\usepackage{amsfonts}
\usepackage[active]{srcltx}
\usepackage{indentfirst,latexsym,bm,amsmath,pstricks,amssymb,amsthm,amscd}
\usepackage[all]{xy}

\begin{document}
\theoremstyle{plain}
\newtheorem{thm}{Theorem}[section]
\newtheorem{theorem}[thm]{Theorem}
\newtheorem{addendum}[thm]{Addendum}
\newtheorem{lemma}[thm]{Lemma}
\newtheorem{corollary}[thm]{Corollary}
\newtheorem{proposition}[thm]{Proposition}

\newcommand{\mR}{\mathbb{R}}
\newcommand{\mZ}{\mathbb{Z}}
\newcommand{\mN}{\mathbb{N}}
\newcommand{\mC}{\mathbb{C}}
\newcommand{\mP}{\mathbb{P}}
\newcommand{\mQ}{\mathbb{Q}}
\newcommand{\CO}{\mathcal{O}}
\newcommand{\CE}{\mathcal{E}}
\newcommand{\CF}{\mathcal{F}}
\newcommand{\CG}{\mathcal{G}}
\newcommand{\CL}{\mathcal{L}}
\newcommand{\CM}{\mathcal{M}}
\newcommand{\CP}{\mathcal{P}}
\newcommand{\CS}{\mathcal{S}}
\newcommand{\CA}{\mathcal{A}}
\newcommand{\CB}{\mathcal{B}}
\newcommand{\CC}{\mathcal{C}}
\newcommand{\CH}{\mathcal{H}}
\newcommand{\CI}{\mathcal{I}}
\newcommand{\CJ}{\mathcal{J}}
\newcommand{\CN}{\mathcal{N}}
\newcommand{\CR}{\mathcal{R}}
\newcommand{\CZ}{\mathcal{Z}}

\newcommand{\FA}{\mathfrak{A}}
\newcommand{\FB}{\mathfrak{B}}
\newcommand{\FC}{\mathfrak{C}}
\newcommand{\FD}{\mathfrak{D}}

\newcommand{\SA}{\mathscr{A}}
\newcommand{\SB}{\mathscr{B}}
\newcommand{\SC}{\mathscr{C}}
\newcommand{\SD}{\mathscr{D}}
\newcommand{\SU}{\mathscr{U}}

\newcommand{\mr}{\mbox}
\newcommand{\I}{{\bf I}}
\newcommand{\J}{{\bf J}}
\newcommand{\bs}{{\bf s}}
\newcommand{\ts}{{\tilde s}}
\newcommand{\e}{{{\bf e}}}
\newcommand{\m}{{\bf m}}
\newcommand{\ord}{\mathrm{ord}}
\newcommand{\fb}{\mathfrak{b}}

\theoremstyle{definition}
\newtheorem{remark}[thm]{Remark}
\newtheorem{remarks}[thm]{Remarks}
\newtheorem{notations}[thm]{Notations}
\newtheorem{definition}[thm]{Definition}
\newtheorem{claim}[thm]{Claim}
\newtheorem{assumption}[thm]{Assumption}
\newtheorem{assumptions}[thm]{Assumptions}
\newtheorem{property}[thm]{Property}
\newtheorem{properties}[thm]{Properties}
\newtheorem{example}[thm]{Example}
\newtheorem{examples}{Examples}
\newtheorem{conjecture}[thm]{Conjecture}
\newtheorem{questiosns}[thm]{Questions}
\newtheorem{question}[thm]{Question}
\newtheorem{problem}[thm]{Problem}
\numberwithin{equation}{section}
 \newcommand{\Rnm}[1]{\uppercase\expandafter{\romannumeral #1}}

\title{Fractional Dehn twists and modular invariants}
\author{Xiao-Lei Liu}
%
%
\date{}

 \maketitle
{\bf Abstract} {  In this note, we establish a relationship between fractional Dehn twist coefficients of Riemann surface automorphisms and modular invariants of holomorphic families of algebraic curves. Specially, we give a characterization of pseudo-periodic maps  with nontrivial fractional Dehn twist coefficients. We also obtain some uniform lower bounds of non-zero fractional Dehn twist coefficients.}\\



\section{Introduction}


Mapping class group is important in low-dimensional topology, and Dehn twists are the generators of this group. Recently, there are many authors (\cite{HKM07, HKM08,HM15,IK12, KR13, Ro011},...) studied fractional Dehn twist coefficients in various aspects of 3-manifold. The study of fractional Dehn twist coefficients dates at least from Gabai and Oertel \cite{GO89}.

Let $\Sigma_g$ be a closed connected Riemann surface of genus $g\geq2$. The mapping class group $\mbox{Mod}(\Sigma_g)$ of $\Sigma_g$ is the group of isotopy classes of orientation preserving homeomorphism of $\Sigma_g$.  The Nielsen-Thurston classification Theorem says that any mapping class $\phi\in \mbox{Mod}(\Sigma_g)$ is either periodic, pseudo-Anosov, or reducible. The homeomorphism $\phi$ is reducible if there exists  finite simple closed curves $\SC=\{\gamma_1,\ldots,\gamma_r\}$ on $\Sigma_g$ such that the restriction of $\phi$ on  $\Sigma_g\backslash\SC$ is either periodic or pseudo-Anosov. If $\phi\in \mathrm{Mod}(\Sigma_g)$ is periodic, or $\phi$ is reducible and the restriction is periodic, then $\phi$ is said to be {\it pseudo-periodic}. We may assume $\SC$ satisfies the following additional conditions (i) $\gamma_i$ does not bound a disk on $\Sigma_g$, and (ii) $\gamma_i$ is not parallel to $\gamma_j$ if $i\neq j$ (\cite[Lemma 1.1]{MM11}). Such $\SC$ is called an {\it admissible system} of cut curves.

  Given a pseudo-periodic map $\phi$, a sufficiently high power $\phi^m$ preserves each cut curve $\gamma_1,\ldots,\gamma_r$.   Denote by $T_{\gamma_i}$ the (right-hand) Dehn twist of $F$ along $\gamma_i$,  then there is a factorization of $\phi$ into a  commutative product
  $$\phi^m=T_{\gamma_1}^{k_1}\cdots T_{\gamma_r}^{k_r}.$$
 The {\it fractional Dehn twist coefficient} of $\phi$ along $\gamma_i$ is defined to be
  $$c(\phi,\gamma_i)=k_i/m.$$
  This definition does not depend on the choice of $m$, and indeed, $c(\phi^n,\gamma_i)=n\cdot c(\phi,\gamma_i)$ (see \cite[Section 2.2.2]{Li17}).

 If $\phi\in {\mathrm{Mod}}(\Sigma_g)$ is a pseudo-periodic map {\it  of negative twist}, that is, $c(\phi,\gamma)<0$ for each $\gamma\in\SC$, then there exists a  local family $f_\phi:S\to\Delta$ whose monodromy homeomorphism around its central fiber is equal (up to isotopy and conjugation) to $\phi$.   Here, a {\it family} of complex projective curves of genus $g$ is a surjective holomorphic morphism $f:X\to Y$ whose general fiber is a smooth curve of genus $g$, where $X$ is a complex smooth projective surface, and $Y$ is a complex smooth projective curve of genus $b$. Moreover, if $Y=\Delta$ is the unit disk of the complex plane, and only the fiber $F=f^{-1}(0)$ over the origin is singular, then we call $f$ a {\it local family}. Denote by $\bar F$ the {minimal normal crossing model} of $F$, that is, $\bar F$ is normal crossing without redundant $(-1)$-curves, and $\bar F_{\scriptstyle{\mathrm{red}}}$ has at worst ordinary double points as its singularities.

If $f$ is a family of curves, and $F_1,\ldots, F_s$ are all the singular fibers of $f$, then there is a local family $f_{F_k}$ with central fiber $F_k$ for each $k$.
Denote by $\phi_{F_k}$ the monodromy homeomorphism of $f_k$, then $\phi_{F_k}$ is a pseudo-periodic map of negative twist (\cite[Theorem 7.1]{MM11}). We denote by $\SC_{F_k}$ the admissible system of cut curves of $\phi_{F_k}$.

Let $\CM_g$ be the moduli space of smooth curves of genus $g$ over the field of complex numbers $\mC$, and $\Delta_0,\Delta_1,\cdots,\Delta_{[g/2]}$ be the boundary divisors of the Deligne-Mumford compactification $\overline \CM_g$. The family $f$ induces  a holomorphic map from $Y$ to the moduli space $\overline{\CM}_g$:
$$J: Y\to \overline{\CM}_g.$$

 For each
$\mathbb{Q}$-divisor class $\eta$ of the moduli space
$\overline{\CM}_g$, we can define an invariant $\eta(f)=\deg
J^*\eta$ which satisfies the {\it base change property}, i.e., if
$\tilde f:\tilde X\to \tilde Y$ is the pullback fibration of $f$
under a base change $\pi:\tilde Y \to Y$ of degree $d$, then
$\eta(\tilde f)=d\cdot \eta(f)$ (see \cite{Ta10}).

Let $\delta_i$ be the  $\mQ$-divisor class corresponding to $\Delta_i$, and
\begin{equation}
\delta=\delta_0+\delta_1+\cdots+\delta_{[g/2]}.
\end{equation}
The corresponding modular invariants of $f$ are $\delta_i(f)$,$\delta(f)$. Hence $\delta_i(f)=0$ if and only if the image $J(Y)$ of $f$ does not intersect with $\Delta_i$. In particular, if $f:X\to Y$ is isotrivial, then $\delta(f)=0$.

Modular invariants  are basic in the study of fibrations of algebraic surfaces and moduli spaces of algebraic curves, see \cite{Ta10,LT13,Li16,No07}. In arithmetic algebraic geometry, modular invariants are some heights of algebraic curves, and can be used to give uniformity properties of curves, see \cite{LT17}.

The first result of this note is a relationship between fractional Dehn twist coefficients and modular invariants. Precisely,

\begin{theorem}\label{thmmain}
Let $\phi\in {\mathrm{Mod}}(\Sigma_g)$ be a pseudo-periodic map of negative twist.
Then
\begin{equation}\label{eqnthmmain1}
  \delta(f_\phi)=\sum_{\gamma\in \SC}|c(\phi,\gamma)|.
\end{equation}
In particular, if  $f:X\to Y$ is a family of curves, then
\begin{equation}
  \delta(f)=\sum_{k=1}^s \sum_{\gamma\in \SC_{F_k}}|c(\phi_{F_k},\gamma)|.
\end{equation}
\end{theorem}

We say that a singularity $p$ in a semistable curve $F$ is {\it
 of type} $i$ if its partial normalization at $p$ consists of
two connected components of arithmetic genera $i$ and $g-i\geq i$,
for $i>0$, and is connected for $i=0$. The general point of $\Delta_0$ is an irreducible curve with a node $p$ of {\it
$\alpha$-type}, i.e., an ordinary double point, hence the node $p$ is of type $0$. The general point of
$\Delta_i$ corresponds to a semistable curve with only one node of type $i~(i\geq1)$ as the
following figure.

 \setlength{\unitlength}{1mm}
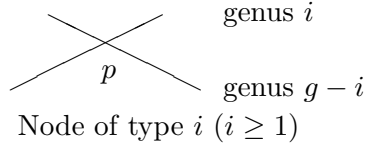
\begin{center}
\begin{picture}(40,20)(0,-5)
\put(-8,-5){\makebox(-10,0)[l]{$~\mbox{~~~~~~Node of type~} i~
(i\geq1)$}} \put(0,0){\line(2,1){20}} \put(25,0){\line(-2,1){20}}
\put(12,2){\makebox(0,0)[l]{$p$}}
\put(28,10){\makebox(0,0)[l]{\mbox{genus}~$i$}}
\put(28,0){\makebox(0,0)[l]{\mbox{genus}~$g-i$}}
\end{picture}
\end{center}

Topologically, if $\gamma$ is a non-separated cut curve, then $\gamma$ is said to be {\it of type 0}; if $\gamma$ is separated, and the least genus  of the two connected components is $i\geq1$, then $\gamma$ is said to be {\it of type $i$}. Set
$$\SC_i=\{\gamma\in\SC:\gamma \mbox{~is~ of~ type~} i\},~~i\geq0.$$

\begin{theorem}\label{thmmaindeltai}
Let $\phi\in {\mathrm{Mod}}(\Sigma_g)$ be a pseudo-periodic map of negative twist.
Then
\begin{equation}
 \delta_i(f_\phi)=\sum_{\gamma\in \SC_i}|c(\phi,\gamma)|.
\end{equation}
In particular, if  $f:X\to Y$ is a family of curves, then
\begin{equation}
 \delta_i(f)=\sum_{k=1}^s \sum_{\gamma\in \SC_{F_k,i}}|c(\phi,\gamma)|.
\end{equation}
\end{theorem}

As a corollary, we can give  a characterization of pseudo-periodic map of negative twist with nontrivial fractional Dehn twist coefficient. See  \cite[Corollary 1.4]{Li17} for another characterization in the aspect of topology of 3-manifold.

\begin{corollary}
Let $\phi$ be a pseudo-periodic map of negative twist. Then all fractional Dehn twist coefficients of $\phi$ vanish, if and only if the image $J(Y)$ of $f_\phi$ by the moduli map $J$ is contained in $\CM_g$.  Moreover, the fractional Dehn twist coefficients $c(\phi,\gamma)$ vanish for all cut curves $\gamma$ of type $i$ if and only if  $J(Y)$ does not intersect with the boundary divisor $\Delta_i$ of $\overline{\CM}_g$.
\end{corollary}

Let $\phi$ be a pseudo-periodic map of negative twist, and $F$ be the central fiber of $f_\phi$. We will show that  each cut curve $\gamma$ corresponds to a unique principal chain $\CC_\gamma$  of $\bar F$, and there is a relation between the data of $\bar F$ with that of $\SC$ in Section \ref{sect-corresp}. Furthermore, we get a formula of $c(\phi,\gamma)$  as follows:

\begin{theorem}[See Theorem \ref{lemCoefScr}]\label{thmCoefScr}
  $$|c(\phi,\gamma)|=\frac{H(\CC_\gamma)}{m_\gamma}.$$
\end{theorem}

For the bounds of fractional Dehn twist coefficients, there are many results in different contexts, see \cite[Theorem 1]{HM15}, \cite[Section 7]{IK12}, \cite[Theorem 2.16]{KR13}. In this note, we will give some uniform lower bounds of fractional Dehn twist coefficients as follows:

\begin{theorem}\label{thmlowerbounds}
Let $\phi\in {\mathrm{Mod}}(\Sigma_g)$ be a pseudo-periodic map of negative twist.
For each cut curve $\gamma$ with $c(\phi,\gamma)\neq0$, we have
\begin{equation}
 |c(\phi,\gamma)|\geq \begin{cases}
\frac1{4g(g+1)^2}, & \mbox{if~} g \mbox{~is~even,}\\
\frac1{4g^2(g+2)}, & \mbox{if~} g \mbox{~is~odd.}
\end{cases}
\end{equation}
Furthermore, if $\gamma$ is a cut curve of type $i$ with $c(\phi,\gamma)\neq0$,
then
\begin{equation}
 |c(\phi,\gamma)|\geq \begin{cases}
\frac1{4g^3}, & \mbox{if~} i=0,\\
\frac1{g(4i+2)(4(g-i)+2)}, & \mbox{if~} i\geq1.
\end{cases}
\end{equation}
\end{theorem}

The modular invariants of families of curves can be generalized to holomorphic foliations and even to differential equations (see \cite{Ta}). It is natural to raise the following question:

\begin{question}
{\it For homeomorphisms which are not  pseudo-periodic map of negative twist,  is there a relation between fractional Dehn twist coefficients with generalized modular invariants?}
\end{question}

\vspace{2mm}
{\bf Notations} If $m,n$ are two integers, we set $\sigma(m,n)$ be the remainder of the division of $n$ by $m$, that is,
$\sigma(m,n)$ is  the integer satisfying $0\leq\sigma(m,n)<n$ and $\sigma(m,n)\equiv m~(\mathrm{mod}~n)$.

 Let ${\bf{a}}=(a_1,a_2,\ldots,a_{r})$ be an ordered sequence of integers, $d({\bf{a}})=\gcd(a_1,\ldots,a_r)$.

 $\lambda_{1,{\bf a}}=\frac{a_1}{d({\bf a})}$, $\lambda_{2,{\bf a}}=\frac{a_r}{d({\bf a})}$,
$\sigma_{1,\bf a}=\sigma(\frac{a_2}{d({\bf a})},\lambda_{1,{\bf a}}),\sigma_{2,\bf a}=\sigma(\frac{a_{r-1}}{d({\bf a})},\lambda_{2,{\bf a}})$.




\section{Modular invariants of families}\label{Section-Modular}

 By blowing up singularities of fibers of $f:X\to Y$, we can obtain a family $\bar f:\bar X\to Y$ satisfying that each singular fiber of $\bar f$ is minimal normal crossing. Let $\pi:\tilde Y\to Y$ be a base change of degree $d$. {\it The pullback fibration $\tilde f$}
of $f$ under $\pi$ is defined as the relative minimal model of the
desingularization of $\bar X\times_Y\tilde Y\to\tilde Y$.  See the following diagram for this construction.
\[\begin{CD}
\tilde{X} @<\tilde{\rho}<< X_2 @>\rho_2>> X_1 @>\rho_1>> X\times_Y \tilde{Y} @>\Pi'>>\bar X\\
@V\tilde{f}VV       @Vf_2VV        @Vf_1VV         @VVV                       @V\bar fVV\\
\tilde{Y} @=      \tilde{Y} @=   \tilde{Y} @=
\tilde{Y} @>\pi>> Y
\end{CD}
\]
Here $\rho_1$ is the normalization, $\rho_2$ is the minimal
desingularization of $X_1$, and $\tilde \rho:X_2\to\tilde X$ is the
contraction of the vertical $(-1)$-curves. We call $\pi$ a {\it
stabilizing base change} if $\tilde f$ is semistable.

Now we consider the above construction locally. Let $F$ be a fiber of $f$ over $p\in Y$. Assume that $\pi$ is totally ramified over $p$, i.e., $\pi^{-1}(p)$ contains only one point $\tilde p$. In this case, $\pi$ is defined  by $z=w^d$ locally, where $z$ (resp. $w$) is the local parameter of $Y$ (resp. $\tilde Y$). Denote by $\tilde F$ the fiber of $\tilde f$ over $\tilde p\in \tilde Y$. If $\tilde F$ is semistable, then $\tilde F$ is called the {\it $d$-th semistable model} of $F$.

 Set $\Pi_2=\Pi'\circ \rho_1\circ \rho_2:X_2\to \bar X$, and $\Pi=\Pi_2\circ\tilde \rho^{-1}:\tilde X\dashrightarrow \bar X$.
 Then $\Pi$ is a well defined rational morphism. For any irreducible smooth component $\tilde C$ of $\tilde F$, we can define the induced morphism $\Pi_{\tilde C}:\tilde C\to C$ by the unique extension, where $C$ is the image. On the other hand, if $C$ is an irreducible component of $\bar F$, we always use $\tilde C$ to denote an irreducible element of $\Pi^{-1}(C)$. Here the element is chosen arbitrarily.

Let $F$ be a singular fiber of $f$, and $\tilde F$ be its $d$-th semistable model.  Denote by $\delta_i(\tilde F)$ the number of nodes of type $i$ in $\tilde F$, then we define
\begin{equation}\label{edeltaF}
\delta_i(F):=\frac{\delta_i(\tilde F)}d,~~~(i=0,1,\ldots,[g/2]),
\end{equation}
 which is independent of the choice of the semistable model $\tilde F$ of $F$. Let $F_1,\ldots,F_s$ be all the singular fibers of $f$. If we choose a stabilizing base change totally ramified over $f(F_1),\ldots,f(F_s)$, then the modular invariant $\delta_i(\tilde f)$ of $\tilde f$ is $\delta_i(\tilde f)=\delta_i(\tilde F_1)+\cdots+\delta_i(\tilde F_s)$ for each $i$ by \cite[Proposition 3.92]{HM98}. So we have that
\begin{equation}\label{edeltaf}
\delta_i(f)=\delta_i(F_1)+\cdots+\delta_i(F_s),~~~i=0,1,\ldots,[g/2].
\end{equation}



 An irreducible component $C$ of $\bar F$ is said to be {\it principal} if either $C$ is not smooth rational, or $C$ meets other components of $\bar F_{\scriptstyle\mathrm{red}}$ at no less than 3 points.

\begin{definition}\label{pchain}
Let $\bar F$ be the minimal normal crossing model of a singular fiber $F$.

(1) An {\it H-J chain} $\CC$ of $\bar F$ is the following subgraph of the dual graph $G(\bar F)$ of $\bar F$,
\vspace{-1cm}
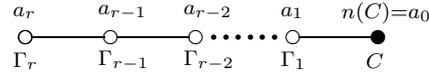
\begin{center}
\setlength{\unitlength}{1mm}
\begin{picture}(50,15)(0,0)
 \setlength{\unitlength}{1.6mm}
 \multiput(0,0)(7,0){3}{\circle{1}}
 \multiput(0.5,0)(7,0){2}{\line(1,0){6}}
 \multiput(15.5,0)(1,0){6}{\circle*{0.5}}
\multiput(22,0)(7,0){2}{\circle{1}}\put(22.5,0){\line(1,0){6}}
\put(0,0){\circle{1}}
\put(29,0){\circle*{1}}

\put(-1,2){\makebox(0,0)[l]{$\scriptstyle{a_r}$}}
\put(-1,-2){\makebox(0,0)[l]{$\scriptstyle{\Gamma_r}$}}
\put(6,2){\makebox(0,0)[l]{$\scriptstyle{a_{r-1}}$}}
\put(6,-2){\makebox(0,0)[l]{$\scriptstyle{\Gamma_{r-1}}$}}
\put(13,2){\makebox(0,0)[l]{$\scriptstyle{a_{r-2}}$}}
\put(13,-2){\makebox(0,0)[l]{$\scriptstyle{\Gamma_{r-2}}$}}
\put(21,2){\makebox(0,0)[l]{$\scriptstyle{a_1}$}}
\put(21,-2){\makebox(0,0)[l]{$\scriptstyle{\Gamma_1}$}}
\put(26,2){\makebox(0,0)[l]{$\scriptstyle{n(C)=a_{0}}$}}
\put(28,-2){\makebox(0,0)[l]{$\scriptstyle{C}$}}
\end{picture}
\end{center}
\vspace{3mm}
where $C$ is a principal component of $\bar F$, $n(C)=\mbox{mult}_{C}(\bar F)$, $a_j=\mbox{mult}_{\Gamma_j}(\bar F)$, $\Gamma_j\cong\mP^1 $ is not principal for each $1\leq j\leq r$, and $\Gamma_r$ meets no other component of $\bar F$.

(2)  If $C_1$ and $C_2$ are two principal components of $\bar F$ ($C_1,C_2$ may be the same). Let $\CC$ be the following subgraph of $G(\bar F)$,

\vspace{-1cm}
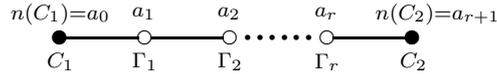
\begin{center}

\setlength{\unitlength}{1mm}
\begin{picture}(50,15)(0,0)
 \setlength{\unitlength}{1.6mm}
 \multiput(0,0)(7,0){3}{\circle{1}}\multiput(0.5,0)(7,0){2}{\line(1,0){6}}
 \multiput(15.5,0)(1,0){6}{\circle*{0.5}}
\multiput(22,0)(7,0){2}{\circle{1}}\put(22.5,0){\line(1,0){6}}
\put(0,0){\circle*{1}}\put(29,0){\circle*{1}}

\put(-4,2){\makebox(0,0)[l]{$\scriptstyle{n(C_1)=a_0}$}}
\put(-1,-2){\makebox(0,0)[l]{$\scriptstyle{C_1}$}}
\put(6,2){\makebox(0,0)[l]{$\scriptstyle{a_1}$}}
\put(6,-2){\makebox(0,0)[l]{$\scriptstyle{\Gamma_1}$}}
\put(13,2){\makebox(0,0)[l]{$\scriptstyle{a_2}$}}
\put(13,-2){\makebox(0,0)[l]{$\scriptstyle{\Gamma_2}$}}
\put(21,2){\makebox(0,0)[l]{$\scriptstyle{a_r}$}}
\put(21,-2){\makebox(0,0)[l]{$\scriptstyle{\Gamma_r}$}}
\put(26,2){\makebox(0,0)[l]{$\scriptstyle{n(C_2)=a_{r+1}}$}}
\put(28,-2){\makebox(0,0)[l]{$\scriptstyle{C_2}$}}
\end{picture}
\end{center}
\vspace{3mm}

\noindent where $n(C_i)=\mbox{mult}_{C_i}(\bar F)$,
$a_j=\mbox{mult}_{\Gamma_j}(\bar F)$, and $\Gamma_j\cong\mP^1$ is not
principal. Then we call $\CC$  a {\it principal chain} between $C_1$ and $C_2$ with {\it multiplicity sequence} ${\bf a}(\CC)=(a_0,a_1,\ldots,a_{r+1})$. If there is no confusion, we set $\CC=\langle C_1,C_2\rangle$.
\end{definition}

 We denote by $PC(\bar F)$ the set of all the principal chains of $\bar F$. If all the nodes of $\Pi^{-1}(p)$ are of type $i$ for any node $p$ of $\CC$, then we call $\CC$ a {\it principal chain of type $i$}. Denote by $PC_i(\bar F)$ the set of all the principal chains of $\bar F$ of type $i$.

 \begin{example}
 The following is a dual graph of a singular fiber of genus 2. There are exactly two principal components $C_1,C_2$. The subgraph between $C_1$ and $C_2$ is a principal chain, and chains between $C_i$ and $\Gamma_j$ ($(i,j)=(1,1),(1,2),(2,3),(2,4)$) are all H-J chains.
  \setlength{\unitlength}{1.5mm}
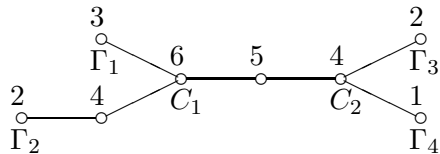
\begin{center}
\begin{picture}(60,15)(-20,-2)

\multiput(7,5)(7,0){3}{\circle{1}}
\put(7.5,5){\line(1,0){6}}\put(14.5,5){\line(1,0){6}}
\put(0,8.5){\circle{1}}\put(0.4,8.3){\line(2,-1){6.3}}
\put(0,1.5){\circle{1}}\put(0.4,1.7){\line(2,1){6.3}}
\put(-7,1.5){\circle{1}}\put(-6.5,1.5){\line(1,0){6}}
\put(28,8.5){\circle{1}}\put(27.6,8.3){\line(-2,-1){6.3}}
\put(28,1.5){\circle{1}}\put(27.6,1.7){\line(-2,1){6.3}}

\put(-1,10.5){\makebox(0,0)[l]{3}}
\put(-1,6.5){\makebox(0,0)[l]{$\Gamma_1$}}
\put(-1,3.5){\makebox(0,0)[l]{4}}
\put(-8,3.5){\makebox(0,0)[l]{2}}
\put(-8,-0.5){\makebox(0,0)[l]{$\Gamma_2$}}
\put(6,7){\makebox(0,0)[l]{6}}
\put(6,3){\makebox(0,0)[l]{$C_1$}}
\put(13,7){\makebox(0,0)[l]{5}}
\put(20,7){\makebox(0,0)[l]{4}}
\put(20,3){\makebox(0,0)[l]{$C_2$}}
\put(27,10.5){\makebox(0,0)[l]{2}}
\put(27,6.5){\makebox(0,0)[l]{$\Gamma_3$}}
\put(27,3.5){\makebox(0,0)[l]{1}}
\put(27,-0.5){\makebox(0,0)[l]{$\Gamma_4$}}
\put(12,0){\makebox(0,0)[c]}
\end{picture}
\end{center}
 \end{example}

\section{Fractional Dehn twists}\label{sectFDT}
   Let $\Sigma$ be a connected real 2-dimensional
manifold with or without boundary. When we emphasize its complex
structure, we call $\Sigma$ a Riemann surface.

Let $\phi:\Sigma\to \Sigma$ be a  pseudo-periodic map. For each cut curve $\gamma$ in the admissible system $\SC$, there exists a minimal integer
$\alpha$ such that $\phi^{\alpha}(\vec{\gamma})=\vec{\gamma}$, i.e., $\phi^{\alpha}(\gamma)=\gamma$ as a set and $\phi^\alpha$
preserving the orientation of $\gamma$.  The curve $\gamma$ is said to be {\it amphidrome} if $\alpha$ is even and
$\phi^{\alpha/2}(\vec{\gamma})=-\vec{\gamma}$, (where $\vec\gamma$ and $-\vec\gamma$ denote the same $\gamma$ with the opposite directions assigned) and {\it non-amphidrome} otherwise. There exists a minimal integer $L$ such that $\phi^{L}$ restricts to an annulus of ${\gamma}$ is isotopic to a Dehn
twist of $e$ times $(e\in \mZ)$.   The rational number $e\alpha/L$ is called the {\it screw number} of $\phi$ about $\gamma$, and is denoted by $s(\gamma)$.
We may always assume that $s(\gamma)\neq0$ for each $\gamma\in\SC$ (see \cite[P5]{MM11}).

Let $\phi:\Sigma\to \Sigma $ be a periodic  homeomorphism  of order $n\geq2$, and $p$ be a point on
$\Sigma$ . There is a positive integer $m_p$ such that the
points $p,\phi(p),\ldots,\phi^{m_p-1}(p)$ are mutually distinct and
$\phi^{m_p}(p)=p$. If $m_p=n$, we call the point $p$ a {\it
simple point} of $\phi$, while if $m_p<n$, we call $p$ a {\it
multiple point} of $\phi$.

 Let $\gamma$ be a cut curve in $\SC$, and $m=m_{\vec{\gamma}}$ be the smallest positive integer such that
$\phi^m(\vec{\gamma})=\vec \gamma$. The restriction of $\phi^m$ to $\vec
\gamma$ is a periodic map of order, say, $\lambda\geq1$.  Let $q$ be any point on  $\gamma$, and suppose that the
images of $q$ under the iteration of $\phi^m$ are ordered
$(q,\phi^{m\sigma}(q),\phi^{2m\sigma}(q),\ldots,
\phi^{(\lambda-1)m\sigma}(q))$ viewed in the direction of $\vec \gamma$,
where $\sigma$ is an integer with $0\leq \sigma \leq \lambda-1$,
$\mathrm{gcd}(\sigma,\lambda)=1$, and $\sigma=0$ iff $\lambda=1$.
Then the action of $\phi^m$ on $\vec \gamma$ is the rotation of angle
$2\pi\delta/\lambda$ with a suitable parametrization of $\vec \gamma$ as
an oriented circle. The triple $(m, \lambda,\sigma)$ is called
the {\it valency} of $\vec \gamma$ with respect to $\phi$.

We define the {\it valency of a boundary curve} (i.e., a
connected component of the boundary $\partial \Sigma$) as its
valency with respect to $\phi$, assuming it has the orientation induced
by the surface $\Sigma$. The {\it valency of a multiple point} $p$
is defined to be the valency of the boundary curve $\partial D_p$, oriented from the outside of a disk neighborhood $D_p$ of $p$.

Suppose $\pi:\Sigma\to \Sigma'$ is the $n$-fold
cyclic covering induced by $\phi$, where $\Sigma'$ is the quotient surface of $\Sigma$ with respect to $\phi$. Let $B_\phi=\{q_1,\ldots,q_s\}\subseteq \Sigma'$ be the
set of branch points. If $\tilde q_i$ is a point of the pre-image
of $q_i$, and let the valency of $\tilde q_i$ be $(m_i,\lambda_i,\sigma_i)$.
Then we know that $m_i$ is the number of points in the pre-image of
$q_i$ and $\lambda_i=n/m_i$. Since the valencies of points in the
pre-image of $q_i$ are the same,  we can define the valency of $q_i$ to be
the valency of $\tilde q_i$.

\section{Correspondence between chains and cut curves}\label{sect-corresp}

Let $f:S\to \Delta$ be a local family of Riemann surfaces of genus $g$, and $\tilde F$ be a semistable model the singular fiber $F$ of $f$. Let  $\phi_f:\Sigma_g\to \Sigma_g$ be the monodromy homeomorphism of $f$.
\begin{lemma}\label{lemmacorrCutPc}
Let $\SC$ be the admissible system of cut curves of $\phi_f$.
Then there is a 1-1 correspondence between $\SC$ and $PC(\tilde F)$ of $\tilde F$.

Moreover, there is a 1-1 correspondence between cut curves $\SC_i$ of type $i$ and principal chains $PC_i(\tilde F)$ of type $i$ of $\tilde F$, for each $i=0,1,\ldots,[g/2]$.
\end{lemma}
\begin{proof}
    By shrinking each curve $\gamma$ of $\SC$ to a point, say $p_\gamma$, we naturally associate a stable Riemann surface. The topological structure of the stable Riemann surface  coincides with the moduli point $F^{\mathrm{st}}$ of $f:S\to \Delta$
    ({\cite[Lemma 4.2]{AI02}}), where $F^{\mathrm{st}}$ is the stable model of $F$. Note that $\tilde F$ has no H-J chains. The stable curve $F^{\mathrm{st}}$ is obtained by contracting $(-2)$-curves in $\tilde F$, hence every  principal chain of $\tilde F$ is contracted to a point in $F^{\mathrm{st}}$.  Let $\tilde\CC_\gamma\in PC(\tilde F)$ be the principal chain contracted to the point $p_\gamma$ in $F^{\mathrm{st}}$. Then $\gamma\mapsto\tilde\CC_\gamma$ gives a 1-1 correspondence between cut curves $\SC$ and principal chains $PC(\tilde F)$ of $\tilde F$.

    Since the stable Riemann surface and $F^{st}$ are the same, the above correspondence preserves the type, by the definition of the type of principal chains of $\tilde F$ and that of the type of cut curves.
\end{proof}

In the following, for each $\gamma\in \SC$, we denote by $\tilde\CC_\gamma\in PC(\tilde F)$ the corresponding principal chain (Lemma \ref{lemmacorrCutPc}).

\begin{theorem}\label{thmcorrCutPrc}
There is a correspondence between $PC(\bar F)$ and the admissible system $\SC$ of cut curves of $\phi_f$, satisfying that:

  1) each $\CC\in PC(\bar F)$ corresponds to  a cyclic orbit $\gamma_1,\ldots, \gamma_{m}$ in $\SC$ under the permutation caused by $\phi_f$, where
  \begin{equation}\label{eqcorrCutPc}
  \Pi^{-1}(\CC):=\{\tilde\CC\in PC(\tilde F): \Pi(\tilde\CC)=\CC\}=\{\tilde\CC_{\gamma_1},\tilde\CC_{\gamma_2},\ldots,\tilde\CC_{\gamma_{m}}\}.
\end{equation}

  2) each $\gamma\in\SC$ corresponds to a unique principal chain $\CC_\gamma\in PC(\bar F)$, with $\CC_\gamma=\Pi(\tilde C_\gamma)$.

  Moreover, the correspondence preserves the type of cut curves and of principal chains. Precisely,

  1i) each $\CC\in PC_i(\bar F)$ corresponds to  a cyclic orbit $\gamma_1,\ldots, \gamma_{m}$ in $\SC_i$,

   2i) each $\gamma\in\SC_i$ corresponds to a unique principal chain  $\CC_\gamma\in PC_i(\bar F)$, where $\CC_\gamma=\Pi(\tilde C_\gamma)$.
\end{theorem}
\begin{proof}
For every principal chain $\tilde \CC\in PC(\tilde F)$ of $\tilde F$,  $\Pi(\tilde\CC)\in PC(\bar F)$ is a principal chain of $\bar F$. For each $\CC\in PC(\bar F)$, its pre-image $\Pi^{-1}(\CC)$ consists of principal chains in $PC(\tilde F)$. The pre-image $\Pi^{-1}(\CC)=\{\tilde\CC\in PC(\tilde F): \Pi(\tilde\CC)=\CC\}$ is a cyclic orbit induced by the base change $\pi:\tilde Y\to Y$, by the construction of semistable reduction, see Section \ref{Section-Modular}. By the definitions of the type of principal chains of $\bar F$ and $\tilde F$, the correspondence preserves the type. Then the result is directly from Lemma \ref{lemmacorrCutPc}.
\end{proof}

Let $\gamma\in\SC$ be a cut curve, and $\gamma_1=\gamma,\gamma_2,\ldots,\gamma_{m_\gamma}$ be the cyclic orbit of $\gamma$ under the permutation caused by $\phi_f$. Then we call $m_\gamma$ the {\it multiplicity} of $\gamma$, which is the length of the orbit. So,  by Theorem \ref{thmcorrCutPrc},
\begin{equation}\label{eq-m-pre}
  m_\gamma=\# \Pi^{-1}(\CC_\gamma).
\end{equation}
 It is easy to see that
 $\CC_\gamma=\Pi(\tilde\CC_{\gamma_1})=\cdots=\Pi(\tilde\CC_{m_\gamma}).$
Let $\CC\in PC(\bar F)$ be a principal chain in Definition \ref{pchain}, we define

\begin{equation}\label{eqH}
H(\CC):=
\sum_{i=0}^r\frac{\gcd(\gamma_i,\gamma_{i+1})^2}{\gamma_i\gamma_{i+1}}.
\end{equation}

For the pseudo-periodic map $\phi_f$, we may assume that (i) there exists a system of disjoint annular neighborhoods $\{A_{\gamma_i}\}_{i=1}^r$ of the admissible system of cut curves subordinate to $\phi$, such that $\phi(\SA)=\SA$ where $\SA=\cup_{i=1}^rA_{\gamma_i}$, (ii) $\phi|_{\SB}:\SB\to \SB$ is a periodic map, where $\SB=\Sigma_g-\mathrm{Int}(\SA)$ (\cite[Theorem 2.1]{MM11}).
\begin{theorem}\label{thmcorrespondence}
Let $\gamma\in \SC$ be a cut curve of $\phi_f$, and the multiplicity sequence of  $\CC_\gamma=\langle C_{1},C_{2} \rangle \in PC(\bar F)$ be ${\bf a}_\gamma=(a_0,a_1,\ldots$, $a_{r+1})$.

  (1) If $\gamma$ is non-amphidrome,  then the valencies of the two boundaries of the annulus $A_{\gamma}$ are
  $$\Big(d({\bf a}_\gamma),\lambda_{1,{\bf a_\gamma}}, \sigma_{1,{\bf a_\gamma}}\Big),\Big(d({\bf a}_\gamma), \lambda_{2,{\bf a_\gamma}}, \sigma_{2,{\bf a_\gamma}}\Big).$$
  Moreover, $m_\gamma=d({\bf a}_\gamma)$, and $|s(\gamma)|=H(\CC_\gamma)$.

  (2) If $\gamma$ is amphidrome,  then we may assume $\mathrm{mult}_{C_2}(\bar F)=d({\bf a}_\gamma)$ and the valencies of the two boundaries of the annulus $A_{\gamma}$ are the same, which is
  $$\Big(d({\bf a}_\gamma), \lambda_{1,{\bf a_\gamma}},  \sigma_{1,{\bf a_\gamma}}\Big).$$  Moreover, $m_\gamma=\frac12 d({\bf a}_\gamma)$, and $|s(\gamma)|=2H(\CC_\gamma)$.

\end{theorem}
\begin{proof}
Step 1: We describe the construction of the numerical topological space $\bar F_{\mathrm{top}}$ from the pseudo-periodic map $\phi:\Sigma_g\to \Sigma_g$ according to \cite{MM11}. Here a numerical topological space is a topological space attached a positive integer to each irreducible component, and the attached positive integer is the multiplicity of the irreducible component. \\

 Now we take a disjoint union $\SD$ of invariant disk neighborhoods  of all the multiple points in $\SB$, and let $\SB'$ denote $\SB-\mathrm{Int}(\SD)$. We further classify the annuli in $\SA$ into $\SA_{\mathrm{sp}}$ and $\SA_{\mathrm{ln}}$ according to their character of being amphidrome or non-amphidrome.

It is proved that ({\cite[Theorem 7.4]{MM11}}): {\it There exist numerical topological spaces $Ch(\SB'),Ch(\SD)$, $Ch(\SA_{\mathrm{ln}}),Ch(\SA_{\mathrm{sp}})$ such that (\cite[P92]{MM11})}
\begin{equation}\label{FCh}
\bar F=Ch(\SB')\cup Ch(\SD)\cup Ch(\SA_{\mathrm{ln}})\cup Ch(\SA_{\mathrm{sp}})
\end{equation}
as numerical topological spaces, and the mentioned spaces are defined as follow:

$\underline{Ch(\SB')}$: The action of $\phi|_{\SB'}$ is free, and the well-defined quotient space $\SB'/(\phi|_{\SB'})$ is the topological space of the numerical surface $Ch(\SB')$. The attached multiplicity of each connected component $P$ of $\SB'/(\phi|_{\SB'})$ is the number of the sheets over $P$.

$\underline{Ch(\SD)}$: Let $p$ be a multiple point of $\phi$, and $p_1=p,p_2,\ldots,p_{m_p}$ be the cyclic orbit of $p$ under the permutation caused by $\phi$. Let $\SD_p$ consist of the components of $\SD$ containing $p_i$ for some $1\leq i\leq m_p$. Then there is a numerical space $Ch(\SD_p)$ shown in Figure \ref{SDp},  and attached the multiplicities $m_p b_0,m_p b_1,\ldots,m_p b_l$
to the components $D_0,S_1,\ldots,S_l$ where $D_0$ is a disk and $S_i$ are spheres.
\begin{figure}[h]
\vspace{5mm}
 \setlength{\unitlength}{1.5mm}
\begin{center}
\begin{picture}(60,15)(-5,-3)
\put(-0.1,5){\oval(10,10)[r]}
\put(9.5,5){\circle{10}}
\put(19,5){\circle{10}}
\multiput(28,5)(1.5,0){3}{\circle*{0.5}}
\put(40,5){\circle{10}}
\put(0,12){\makebox(0,0)[l]{\scriptsize$D_0$}}
\put(9,12){\makebox(0,0)[l]{\scriptsize$S_1$}}
\put(18.5,12){\makebox(0,0)[l]{\scriptsize$S_2$}}
\put(39.5,12){\makebox(0,0)[l]{\scriptsize$S_l$}}
\put(-2,5){\makebox(0,0)[l]{\scriptsize$m_pb_0$}}
\put(7.5,5){\makebox(0,0)[l]{\scriptsize$m_pb_1$}}
\put(17,5){\makebox(0,0)[l]{\scriptsize$m_pb_2$}}
\put(37.5,5){\makebox(0,0)[l]{\scriptsize$m_pb_l$}}

\end{picture}
\end{center}
\vspace{-5mm}
\caption{Numerical space $Ch(\SD_p)$\label{SDp}}
\end{figure}
 Here the sequence of integers $b_0>b_1>\cdots>b_l=1$ satisfies
$b_{j-1}+b_{j+1}\equiv 0~(\mathrm{mod}~ b_j)~~(j=0,1,\ldots,l-1),$
and $(m_p,b_0,b_1)$ is the valency of $p$ (Lemma 3.2 in \cite{MM11}). The space $Ch(\SD)$ is the union of $Ch(\SD_p)$, where $p$ runs over all the multiple points.

$\underline{Ch(\SA_{\mathrm{ln}})}$: Let $\gamma$ be a non-amphidrome cut curve, and $\gamma_1=\gamma,\gamma_2,\ldots,\gamma_{m_\gamma}$ be the cyclic orbit of $\gamma$ under the permutation caused by $\phi$. Let $\SA_\gamma=\cup_{i=1}^{m_\gamma} A_{\gamma_i}$ be the components of $\SA$ containing $\gamma_i$. The space $Ch(\SA_\gamma)$ is shown in Figure \ref{Aln}, and attached the multiplicities
$m_\gamma b_0,m_\gamma b_1,\ldots,m_\gamma b_l$
to the components $D_0,S_1,\ldots,S_{l-1},D_l$ where $D_i~(i=0,l)$ are disks and $S_i$ are spheres.
\begin{figure}[h]
\vspace{5mm}
 \setlength{\unitlength}{1.5mm}
\begin{center}
\begin{picture}(60,15)(-5,-3)
\put(-0.1,5){\oval(10,10)[r]}
\put(9.5,5){\circle{10}}
\put(19,5){\circle{10}}
\multiput(28,5)(1.5,0){3}{\circle*{0.5}}
\put(40,5){\circle{10}}
\put(49.6,5){\oval(10,10)[l]}
\put(0,12){\makebox(0,0)[l]{\scriptsize$D_0$}}
\put(9,12){\makebox(0,0)[l]{\scriptsize$S_1$}}
\put(18.5,12){\makebox(0,0)[l]{\scriptsize$S_2$}}
\put(38,12){\makebox(0,0)[l]{\scriptsize$S_{l-1}$}}
\put(48.7,12){\makebox(0,0)[l]{\scriptsize$D_l$}}
\put(-2,5){\makebox(0,0)[l]{\scriptsize$m_\gamma b_0$}}
\put(7.5,5){\makebox(0,0)[l]{\scriptsize$m_\gamma b_1$}}
\put(17,5){\makebox(0,0)[l]{\scriptsize$m_\gamma b_2$}}
\put(37,5){\makebox(0,0)[l]{\scriptsize$m_\gamma b_{l-1}$}}
\put(47.5,5){\makebox(0,0)[l]{\scriptsize$m_\gamma b_l$}}

\end{picture}
\end{center}
\vspace{-5mm}
\caption{Numerical space $Ch(\SA_{\gamma})$\label{Aln}~($\gamma$ non-amphidrome)}
\end{figure}
 Here the sequence of integers $b_0,b_1,\cdots,b_l$ satisfies
$b_{i-1}+b_{i+1}\equiv0~(\mathrm{mod} ~b_i)$, ${b_{i-1}+b_{i+1}}\geq2{b_i},~(i=1,2,\ldots,l-1),$
\begin{equation}
\sum_{i=0}^{l-1}\frac1{b_ib_{i+1}}=|s|,
\end{equation}
and  $(m_\gamma,b_0,\sigma(b_1,b_0)),(m_\gamma,b_l,\sigma(b_{l-1},b_l))$ are the valencies of the two boundaries of $A_\gamma$ (Lemma 3.3 in \cite{MM11}). The space $Ch(\SA_{\mathrm{ln}})$ is the union of $Ch(\SA_\gamma)$, where $\gamma$ runs over all the non-amphidrome cut curves.

 In this case, we also denote  $Ch(\SA_\gamma)$ by  $Ch(\SA_\gamma)'$ for simplicity.

$\underline{Ch(\SA_{\mathrm{sp}})}$: Let $\gamma$ be a amphidrome cut curve, and $\gamma_1=\gamma,\gamma_2,\ldots,\gamma_{m_\gamma}$ be the cyclic orbit of $\gamma$ under the permutation caused by $\phi$. Let $\SA_\gamma$ be the components of $\SA$ containing $\gamma_i$. The space $Ch(\SA_\gamma)$ is shown in Figure \ref{Asp}, and attached the multiplicities
$2m_\gamma b_0,2m_\gamma b_1,\ldots,2m_\gamma b_l,m_\gamma,m_\gamma$
to the components $D_0,S_1,\ldots,S_l,S_1',S_2'$. Here $D_0$ is a disk , $S_i$ and $S_i'$ are spheres.
\begin{figure}[h]
\vspace{5mm}
 \setlength{\unitlength}{1.5mm}
\begin{center}
\begin{picture}(60,15)(-5,-3)
\put(-0.1,5){\oval(10,10)[r]}
\put(9.5,5){\circle{10}}
\put(19,5){\circle{10}}
\multiput(28,5)(1.5,0){3}{\circle*{0.5}}
\put(40.5,5){\circle{10}}
\put(50,5){\circle{10}}
\put(55.3,10.3){\circle{6}}
\put(54.7,-0.7){\circle{6}}
\put(56,13.8){\makebox(0,0)[l]{\scriptsize$S_1'$}}
\put(55.5,3){\makebox(0,0)[l]{\scriptsize$S_2'$}}
\put(54.3,10.3){\makebox(0,0)[l]{\scriptsize$m_\gamma$}}
\put(53.7,-0.7){\makebox(0,0)[l]{\scriptsize$m_\gamma$}}
\put(0,12){\makebox(0,0)[l]{\scriptsize$D_0$}}
\put(9,12){\makebox(0,0)[l]{\scriptsize$S_1$}}
\put(18.5,12){\makebox(0,0)[l]{\scriptsize$S_2$}}
\put(38,12){\makebox(0,0)[l]{\scriptsize$S_{l-1}$}}
\put(48.7,12){\makebox(0,0)[l]{\scriptsize$S_l$}}
\put(-2,5){\makebox(0,0)[l]{\scriptsize$m_\gamma  b_0$}}
\put(7.5,5){\makebox(0,0)[l]{\scriptsize$m_\gamma b_1$}}
\put(17,5){\makebox(0,0)[l]{\scriptsize$m_\gamma b_2$}}
\put(37,5){\makebox(0,0)[l]{\scriptsize$m_\gamma b_{l-1}$}}
\put(47.5,5){\makebox(0,0)[l]{\scriptsize$m_\gamma b_l$}}

\end{picture}
\end{center}
\vspace{-5mm}
\caption{Numerical space $Ch(\SA_{\gamma})$\label{Asp}~($\gamma$ amphidrome)}
\end{figure}
Here the sequence of integers $b_0,b_1,\cdots,b_{l-1}$, $b_l=1$ satisfies
$b_{i-1}+b_{i+1}\equiv0~(\mathrm{mod} ~b_i)$, $(i=1,2,\ldots,l-1),$
\begin{equation}
  \sum_{i=0}^{l-1}\frac1{b_ib_{i+1}}=\frac12|s|,
\end{equation}
and  $(2m_\gamma,b_0,\sigma(b_1,b_0))$ is the valency of the boundary of $A_\gamma$ (Lemma 3.4 in \cite{MM11}). The space $Ch(\SA_{\mathrm{sp}})$ is the union of $Ch(\SA_\gamma)$, where $\gamma$ runs over all the amphidrome cut curves.

In this case, we denote by $Ch(\SA_\gamma)'$ the subspace of $Ch(\SA_\gamma)$ consisting of $D_0$,$S_1,\ldots$,$S_l$ with their multiplicities.\\

Step 2:  Since $\SC$ is the admissible system of cut curves, each connected component $P$ in $Ch(\SB')$ is a principal component of $\bar F$, and each  disk $D_i~(i=0,l)$ in $Ch(\SD)\cup Ch(\SA_{\mathrm{ln}})\cup Ch(\SA_{\mathrm{sp}})$ is part of some connected component $P$ in $Ch(\SB')$. Hence, by Definition \ref{pchain}, $Ch(\SA_\gamma)'$  is the same with the principal chain $\CC_\gamma\in PC(\bar F)$ as a numerical topological space, for each cut curve $\gamma$.

If $\gamma$ is amphidrome, then $\CC_\gamma=\langle C_{1},C_{2} \rangle \in PC(\bar F)$ is the dual graph of Figure \ref{Asp}, where we may assume $D_0$ is part of $C_1$ and $S_l=C_2$. Hence $d(\CC_\gamma)=\mathrm{mult}_{C_2}(\bar F)$.

Other results are directly from the construction of $Ch(\SA_{\mathrm{ln}})$ and $Ch(\SA_{\mathrm{sp}})$ above.

\end{proof}

\begin{remark}
 Let $\CC_{\mathrm{HJ}}$ be a H-J chain with multiplicity $(a_0,a_1,\ldots,a_r)$ where  $a_r|a_0$. From the proof above, either $\CC_{\mathrm{HJ}}$ is a part of $Ch(\SA_{\mathrm{sp}})$, or there exists a multiple point $p$ such that $\CC_{\mathrm{HJ}}$ is the same with $Ch(\SD_p)$ as a numerical topological space. In the latter case, the valency of the multiple point $p$ is
  $$\Big(a_r,\frac{a_0}{a_r},\sigma(\frac{a_1}{a_r},\frac{a_0}{a_r})\Big).$$
\end{remark}

\begin{theorem}\label{lemCoefScr}
Let $\phi$ be a pseudo-periodic map of negative twist, $\gamma\in \SC$, and $\CC_\gamma\in PC(\bar F)$ be the corresponding principal chain of the fiber $\bar F$ of $f_\phi$. Then

$$
|c(\phi,\gamma)|=\frac{H(\CC_\gamma)}{m_\gamma}=
\begin{cases}
\frac{|s_\gamma|}{m_\gamma}, & \mbox{if~} \gamma \mbox{~is~non-amphidrome,}\\
\frac{|s_\gamma|}{2m_\gamma}, & \mbox{if~} \gamma \mbox{~is~amphidrome.}
\end{cases}
$$
\end{theorem}
\begin{proof}

 Using the above notations, $\gamma_1=\gamma,\gamma_2,\ldots,\gamma_{m_\gamma}$ is the cyclic orbit of $\gamma$,  $A_{\gamma_i}$ are disjoint annuli, and $\SA_\gamma=\cup_{i=1}^{m_\gamma} A_{\gamma_i}$. Hence  $\phi|_{\SA_\gamma}:\SA_\gamma\to\SA_\gamma$ is a homeomorphism, and the restriction of $\phi$ to the boundary $\partial \SA_\gamma=\cup_{i=1}^{m_\gamma} \partial A_{\gamma_i}$ is periodic. We know that $m_\gamma$ is the smallest positive integer such that (i) $\phi^{m_\gamma}(A_{\gamma_i})=A_{\gamma_i}$.

  If $\gamma$ is amphidrome, then $\alpha_\gamma=2m_\gamma$ is the smallest positive integer satisfying that (i) and (ii) $\phi^{\alpha_\gamma}$ does not interchange the boundary components. If $\gamma$ is nonamphidrome, then $\alpha_\gamma=m_\gamma$ is the smallest positive integer satisfying that (i) and (ii). See Theorem 2.3 and Theorem 2.4 in \cite{MM11}.

   Let $l_\gamma$ be a non-zero integer such that $\phi^{l_\gamma}|_{\partial \SA_\gamma}$ is the identity. Then $l_\gamma$ is a multiple of $\alpha_\gamma$, and $\phi^{l_\gamma}:A_i\to A_i$ is the result of $e_\gamma$ full Dehn twists, where $e_\gamma$ is an integer.  Then (\cite[P22]{MM11}) $$s(A_i):=e_\gamma \alpha_\gamma/l_\gamma=s(\gamma), ~~i=1,2,\ldots,m_\gamma.$$
Thus
  $$|c(\phi,\gamma)|=\frac{|e_\gamma|}{l_\gamma}=\frac{|s_\gamma|}{\alpha_\gamma}=
\begin{cases}
\frac{|s_\gamma|}{m_\gamma}, & \mbox{if~} \gamma \mbox{~is~non-amphidrome,}\\
\frac{|s_\gamma|}{2m_\gamma}, & \mbox{if~} \gamma \mbox{~is~amphidrome.}
\end{cases}$$
Hence by Theorem \ref{thmcorrespondence}, $|c(\phi,\gamma)|=H(\CC_\gamma)/m_\gamma$.
\end{proof}

\section{Proof of Theorems in Introduction}

\begin{proof}[Proof of Theorem \ref{thmmain}]
  From \cite[Lemma 3.1]{LT17}, we have that
  $$\delta(F)=\sum_{i=0}^{[g/2]}\delta_i(F)=\sum_{\CC\in PC(\bar F)}H(\CC).$$
 By (\ref{eq-m-pre}),  Theorem \ref{lemCoefScr}, and correspondences in Section \ref{sect-corresp}, we have that
  \begin{align*}
    \delta(F)&=\sum_{\CC\in PC(\bar F)}H(\CC)=\sum_{\CC_\gamma\in PC(\bar F)}\sum_{\tilde\CC\in\Pi^{-1}(\CC_\gamma)}\frac{H(\CC_\gamma)}{\#\Pi^{-1}(\CC_\gamma)}\\
    &=\sum_{\CC_\gamma\in PC(\bar F)}\sum_{\tilde\CC\in\Pi^{-1}(\CC_\gamma)}\frac{H(\CC_\gamma)}{m_\gamma}=\sum_{\CC_\gamma\in PC(\bar F)}\sum_{\tilde\CC\in\Pi^{-1}(\CC_\gamma)}|c(\phi,\gamma)|\\
    &=\sum_{\gamma\in\SC}|c(\phi,\gamma)|.
  \end{align*}
\end{proof}

\begin{proof}[Proof of Theorem \ref{thmmaindeltai}]
  From \cite[Lemma 3.1]{LT17}, we have that
  $$\delta_i(F)=\sum_{\CC\in PC_i(\bar F)}H(\CC).$$
   Then the proof is similar to that of Theorem \ref{thmmain}, by Lemma \ref{lemmacorrCutPc}.
\end{proof}

\begin{proof}[Proof of Theorem \ref{thmlowerbounds}]
By the definition of admissible system $\SC=\{\gamma_i\}_{i=1}^r$ of cut curves, we know that $g\geq \#\SC=r$. Thus $g\geq m_\gamma$ for each $\gamma\in\SC$, for the cyclic orbit $\gamma_1,\ldots, \gamma_{m_\gamma}$ is a subset of $\SC$.

    If $\gamma$ is a cut curve of type 0, thus by  the proof of \cite[Theorem 1.4]{LT17} and Theorem \ref{lemCoefScr}, we have that
  $$|c(\phi,\gamma)|=\frac{H(\CC_\gamma)}{m_\gamma}\geq \frac1{4g^3}.$$
   Similarly, if $\gamma$ is a cut curve of type $i\geq1$, then
  $$|c(\phi,\gamma)|=\frac{H(\CC_\gamma)}{m_\gamma}\geq \frac1{g(4i+2)(4(g-i)+2)}.$$

  Hence for any $\gamma\in\SC$, we have that
  \begin{align*}
    |c(\phi,\gamma)|&\geq \mathrm{min}\big\{\frac1{4g^3}, \mathrm{min}\{\frac1{g(4i+2)(4(g-i)+2)}:i=1,2,\ldots, [g/2]\}\big\}\\
    &\geq\frac1{g(4[g/2]+2)(4(g-[g/2])+2)}\\
    &\geq
    \begin{cases}
\frac1{4g(g+1)^2}, & \mbox{if~} g \mbox{~is~even,}\\
\frac1{4g^2(g+2)}, & \mbox{if~} g \mbox{~is~odd.}
\end{cases}
  \end{align*}

\end{proof}

The above uniform lower bounds of fractional Dehn twist coefficients for each $g\geq2$  are not strict. It is interesting to obtain sharp lower bounds for these coefficients.

\section*{Acknowledgement}

 The author would like to thank Prof. Shengli Tan and Prof. Jun Lu for their interesting and fruitful discussions. Thanks to Prof. Youlin Li, Xiaoming Du and Yumin Zhong for their discussions on Dehn twist. The author is very grateful to Prof. Tadashi Ashikaga and Prof. Yukio Matsumoto for useful comments on pseudo-periodic maps. The author thanks the referee for pointing out mistakes. This work is supported by NSFC (No. 11601504) and Fundamental Research  Funds of the Central Universities (No. DUT18RC(4)065).


School of Mathematical Sciences

Dalian University of Technology

Dalian 116024, PR China

 E-mail: xlliu1124@dlut.edu.cn
\end{document}